\documentclass[11pt,a4paper]{amsart}

\usepackage{amsthm,amsmath,amsfonts,mathrsfs,amssymb,stmaryrd}
\usepackage{enumitem}
\usepackage{algorithm}
\usepackage{algpseudocode}
\usepackage[T1]{fontenc}   
\usepackage{multirow}
\usepackage{hyperref}
\usepackage{url}
\usepackage{tikz}
\usetikzlibrary{matrix,arrows,calc}

\usepackage{array}
\newcolumntype{L}[1]{>{\raggedright\let\newline\\\arraybackslash\hspace{0pt}}m{#1}}
\newcolumntype{C}[1]{>{\centering\let\newline\\\arraybackslash\hspace{0pt}}m{#1}}
\newcolumntype{R}[1]{>{\raggedleft\let\newline\\\arraybackslash\hspace{0pt}}m{#1}}

\usepackage{todonotes}

\theoremstyle{theoremstyle}
\newtheorem{theorem}{Theorem}
\newtheorem{lemma}[theorem]{Lemma}
\newtheorem{proposition}[theorem]{Proposition}
\newtheorem{corollary}[theorem]{Corollary}
\theoremstyle{definition}
\newtheorem{definition}[theorem]{Definition}
\newtheorem{example}[theorem]{Example}
\newtheorem{remark}[theorem]{Remark}

\newcommand{\ma}{\mathfrak{a}}

\newcommand{\ini}{\text{in}}
\newcommand{\mI}{\mathfrak{I}}

\newcommand{\bigslant}[2]{{\raisebox{.2em}{$#1$}\left/\raisebox{-.2em}{$#2$}\right.}}
\newcommand{\suchthat}{\;\ifnum\currentgrouptype=16 \middle\fi|\;}

\DeclareMathOperator{\ord}{ord}
\DeclareMathOperator{\initial}{in}

\DeclareMathOperator{\LM}{LM}
\DeclareMathOperator{\Mon}{Mon}

\DeclareMathOperator{\NF}{NF}
\DeclareMathOperator{\Hilb}{H}
\DeclareMathOperator{\HilbP}{P}
\DeclareMathOperator{\HilbS}{HS}
\DeclareMathOperator{\HilbSP}{HSP}
\DeclareMathOperator{\N}{\mathbb{N}}
\DeclareMathOperator{\Z}{\mathbb{Z}}
\DeclareMathOperator{\Q}{\mathbb{Q}}
\DeclareMathOperator{\R}{\mathbb{R}}

\DeclareMathOperator{\A}{\mathbb{A}}

\DeclareMathOperator{\bK}{K}
\DeclareMathOperator{\bH}{{\bf H}}

\DeclareMathOperator{\Quot}{Quot}
\DeclareMathOperator{\Gr}{Gr}
\DeclareMathOperator{\syz}{syz}

\title[Algorithms for computing in localizations at prime ideals]%
{Mora's holy grail$^{1}$: Algorithms for computing in localizations at prime ideals}

\author{Magdaleen S. Marais}
\address{Magdaleen S. Marais\\
African Institute for Mathematical Sciences and University of Pretoria\\
Department of Mathematics and Applied Mathematics\\
0002 Pretoria\\
South Africa}
\email{magdaleen@aims.ac.za}

\author{Yue Ren}
\address{Yue Ren\\
Department of Mathematics\\
University of Kaiserslautern\\
Erwin-Schr\"odinger-Str.\\
67663 Kaiserslautern\\
Germany}
\email{ren@mathematik.uni-kl.de}

\thanks{This research was supported by the African Institute for Mathematical Sciences, the Department of Mathematics and Applied Mathematics of the University of Pretoria and a grant awarded by Wolfram Decker. We are thankful to all of them.}

\begin{document}

\begin{abstract}
This article discusses a computational treatment of the localization $A_L$ of an affine coordinate ring $A$ at a prime ideal $L$ and its associated graded ring $\Gr_\ma(A_L)$ with the means of standard basis techniques. Building on Mora's work \cite{M1991}, we present alternative proofs on two of the central statements and expand on the applications mentioned by Mora: resolutions of ideals, systems of parameters and Hilbert polynomials, as well as dimension and regularity of $A_L$. All algorithms are implemented in the library {\tt graal.lib} for the computer algebra system \textsc{Singular}.
\end{abstract}

\maketitle

\section{Introduction}
\footnotetext[1]{see \cite{M1991}.}
Computer algebra systems like \textsc{Macaulay2} and \textsc{Singular} have become a staple in studying affine or projective varieties globally, by examining their coordinate rings or homogeneous coordinate rings. 

The behaviour of an affine variety $X\subseteq k^n$ around a single point $p\in X$ is described by the localization $A(X)_{\mathfrak m_p}$ of its coordinate ring $A(X)$ at the corresponding maximal ideal $\mathfrak{m}_p$. This local ring is commonly realized by applying an affine coordinate transformation $\varphi:k^n \overset{\sim}{\longrightarrow} k^n$, shifting the point $p$ into the origin $0$. The respective localization $A(\varphi(X))_{\mathfrak{m}_{0}}$ can then be simulated by working with a local monomial ordering. This technique has been applied in the study of isolated singularities to great success.

In this article, we discuss a related approach for the realization of the localization $A_L$ of an affine coordinate ring $A$ at a prime ideal $L$, see Lemma~\ref{lem:Al}. Introduced by Mora in \cite{M1991}, it uses a local ordering on a specific ring to mimic the local structure of $A_L$. From it, one can also deduce a corresponding representation of its associated graded ring $\Gr_\ma(A_L)$, $\ma\unlhd A_L$ denoting its maximal ideal, see Lemma~\ref{lem:graal}. In Section~\ref{HolyGrail}, we discuss the aforementioned Lemmata, giving alternative proofs than in \cite{M1991}.


Because a graded ring $\Gr_\ma(A_L)$ is computationally easier to handle than a local ring $A_L$, we show in Section~\ref{Applications} how to exploit the $\ma$-adic topology on $A_L$, allowing us to work over $\Gr_\ma(A_L)$ instead. Albeit not viable for every problem, a lot of information about $A_L$ can be found preserved in $\Gr_\ma(A_L)$. The problems covered are the Hilbert polynomials of ideals in $A_L$, a system of parameters for $A_L$ as well as the regularity and dimension of $A_L$. In particular, as stated in \cite{M1991}, we present an algorithm for lifting a resolution of an initial ideal $\initial_\ma(I)\unlhd\Gr_\ma(A_L)$ to a resolution of the ideal $I\unlhd A_L$.

All algorithms and examples in this article have been implemented in the \textsc{Singular} library {\tt graal.lib} \cite{graal}, which also showcases some of the new object-oriented features of \textsc{Singular}. The library will be available in the next \textsc{Singular} distribution.

\section{An Algorithm for the Associated Graded Ring $\Gr_a(A_L)$}\label{HolyGrail}

In this article, let $Q=k[X]$ be a multivariate polynomial ring over a ground field $k$, and $H\subseteq J\unlhd Q$ two prime ideals. Let $A:=Q/H$ be the affine coordinate ring of the variety defined by $H$, and let $L:=J\cdot A\unlhd A$ be a prime ideal describing an irreducible subvariety.

Let $U\subseteq X$ be a maximal independent set of variables with respect to $J$, and let $V:=X\setminus U$ denote the remaining variables, so that $J\cdot k(U)[V] \unlhd k(U)[V]$ becomes a maximal ideal. Set $Q^0:=k(U)[V]$, $H^0=H\cdot Q^0$ and $J^0=J\cdot Q^0$.

Next, fix a system of generators $J=\langle f_1,\ldots,f_s\rangle$, and consider the polynomial ring $Q[Y]:=Q[Y_1,\ldots,Y_s]$. Set $\mI:=\langle f_1-Y_1,\ldots,f_s-Y_s \rangle + H^0 \unlhd Q^0[Y]$.
Let $\ma:=L\cdot A_L$ denote the maximal ideal of $A_L$, $K:=A_L/\ma$ its residue field.

\begin{center}
  \begin{tikzpicture}
    \matrix (m) [matrix of math nodes, row sep=2em, column sep=3em, column 1/.style={anchor=base east},]
    { Q:=k[X] & Q^0:=k(U)[V] & & \ma \unlhd A_L\\
      A:=Q/H  &              & & \phantom{hi}  \\ };
    \draw[->] (m-1-1) -- (m-1-2);
    \draw[->>] ($(m-1-1.south)+(-0.7,0)$) -- ($(m-2-1.north)+(-0.7,0)$);
    \node[anchor=east] (unlhdUp) at (m-1-1.west) {$\unlhd$};
    \node[anchor=east] (J) at (unlhdUp.west) {$J$};
    \node[anchor=south,yshift=0.3cm,font=\footnotesize] (f) at (J.north) {\phantom{p}$\langle f_1,\ldots,f_s\rangle$\phantom{p}};
    \draw[draw opacity=0] (f) -- node[sloped,font=\footnotesize] {$=$} (J);
    \node[anchor=east] (subseteq) at (J.west) {$\subseteq$};
    \node[anchor=east] (H) at (subseteq.west) {$H$};
    \node[anchor=east] (unlhdDown) at (m-2-1.west) {$\unlhd$};
    \node[anchor=east] (L) at (unlhdDown.west) {$L:=J\cdot A$};
    \node (V0) at (m-1-2.19.5) {};
    \node[anchor=south,yshift=0.1cm,font=\footnotesize] (V1) at (V0.north) {$X\setminus U$};
    \draw[draw opacity=0] (V0) -- node[sloped,font=\footnotesize] {$=$} (V1);
    \node (U0) at (m-1-2.320) {};
    \node[anchor=north,yshift=-0.3cm,font=\footnotesize] (U1) at (U0.south) {maximal independent set};
    \node[anchor=north,yshift=0.2cm,font=\footnotesize] (U2) at (U1.south) {with respect to $J$};
    \draw[shorten <= -5pt] (U0) -- (U1);

    \node[xshift=0.25cm] (K) at (m-2-4) {$K$};
    \draw[->>] ($(m-1-4.south)+(0.25,0)$) -- (K.north);

    \node[anchor=south,xshift=-0.5cm,yshift=0.4cm,font=\footnotesize] (maxideal) at (m-1-4.north) {maximal ideal};
    \draw (maxideal.south) -- ($(m-1-4.north)+(-0.5,0)$);
  \end{tikzpicture}
\end{center}

Throughout the entire article, we will be abbreviating $(Y_1,\ldots,Y_s)$ with $Y$, $(f_1,\ldots,f_s)$ with $f$ and also, for $\alpha=(\alpha_1,\ldots,\alpha_s)\in\N^s$, $Y_1^{\alpha_1}\cdots Y_s^{\alpha_s}$ with $Y^\alpha$, $f_1^{\alpha_1}\cdots f_s^{\alpha_s}$ with $f^\alpha$.

Set $w\in\R^{|V|}\times\R^{|Y|}$ to be the weight vector which weights the variables $V$ with $0$ and the variables $Y$ with $-1$. In particular, for a polynomial $f\in Q^0[Y]$, $-\deg_w(f)$ will then be the lowest degree in $Y$ occurring in it. Let $>$ be a weighted ordering on $\Mon(V,Y)$ with weight vector $w$ and any global ordering $>'$ on $\Mon(V)$ as tiebreaker, i.e.
\begin{align*}
  V^\beta\cdot Y^\alpha > V^\delta\cdot Y^\gamma \quad :\Longleftrightarrow \quad \big(|\alpha|<|\gamma|\big) \;\text{ or }\; \big(|\alpha|=|\gamma| \text{ and } V^\beta >' V^\delta\big)
\end{align*}
with $V^\beta>'1$ for all $\beta\in\N^{|V|}$.

Let $Q^0[Y]_>$ denote the localization of $Q^0[Y]$ at the monomial ordering $>$,
\[ Q^0[Y]_> := \left\{ \frac{f}{u} \suchthat f,u\in Q^0[Y] \text{ and } \LM_>(u)=1 \right\}. \]
The notions of leading monomials then extend naturally from $Q^0[Y]$ to $Q^0[Y]_>$.

\begin{lemma}[corresponding to Lemma 6.3 in \cite{M1991}]\label{lem:Al}
  The map $Q[Y] \rightarrow A$, $Y^\alpha\mapsto \overline{f^\alpha}$ extends naturally to a map $\phi: Q^0[Y]_> \rightarrow A_L$, inducing an exact sequence of filtered rings
  \begin{center}
    \begin{tikzpicture}[description/.style={fill=white,inner sep=2pt}]
      \matrix (m) [matrix of math nodes, row sep=0.5em,column sep=0em,text height=1.5ex, text depth=0.25ex]
      { 0 &\longrightarrow& \mI \cdot\, Q^0[Y]_> &\longrightarrow& Q^0[Y]_> &\overset{\phi}{\longrightarrow}& A_L &\longrightarrow& 0. \\
          \\
          \\
          & & & & \langle Y \rangle^d\cdot Q^0[Y]_> & \longmapsto & \ma^d\cdot A_L & \\
          & & & & Y^\alpha & \longmapsto & \overline{f^\alpha} & \\ };
        \draw[draw opacity=0] (m-4-5) -- node[sloped] {$\subseteq$} (m-1-5);
        \draw[draw opacity=0] (m-4-7) -- node[sloped] {$\subseteq$} (m-1-7);
        \node[anchor=north,font=\footnotesize,yshift=-0.35cm,xshift=-0.8cm] (I) at (m-1-3.south) {$\langle f_1-Y_1,\ldots,f_s-Y_s \rangle + H^0$};
        \draw[draw opacity=0] (I) -- node[sloped,font=\footnotesize] {$=$} (m-1-3.200);
    \end{tikzpicture}
  \end{center}
\end{lemma}
\begin{proof}
  Before we show exactness, note that the map $\phi$ is well-defined because the residues of the algebraically independent variables $U$ in $A$ are naturally not contained in $L$, making them invertible in $A_L$. Moreover, any polynomial $g\in Q^0[Y]$ with $\LM_>(g)=1$ has to be of the form $g=c+p$ for some $c\in k(U)$, $p\in\langle Y \rangle$, and hence it is also mapped to something invertible in $A_L$.

  It is clear that $\mI\cdot Q^0[Y]_>$ is the kernel of $\phi$. In order to show that $\phi$ is surjective, it suffices to find elements mapping to $p^{-1}$ for $p\notin L$.

  Let $p\notin L$ and let $g=\sum_{\alpha} c_\alpha \cdot Y^\alpha\in Q^0[Y]_>$, with $c_\alpha\in Q^0$, be a preimage of $p$. Clearly, $c_0\neq 0$ since $Y$ is mapped into $L$. We will now show that $c_0$ is invertible modulo $\mI$, which implies that $g$ is invertible modulo $\mI$ and hence completes the proof. Because $J^0\unlhd Q^0$ is maximal, there exist $a\notin J^0$ and $b\in J^0$ such that $a\cdot c_0=1+b$. And since $J^0=\langle f_1,\ldots,f_s\rangle$ and $\mI\supseteq\langle f_1-Y_1,\ldots,f_s-Y_s\rangle$, $a\cdot c_0=1+p_b \text{ mod } \mI$, for some $p_b\in\langle Y_1,\ldots,Y_s\rangle$. Due to the ordering $>$, $1+p_b$ is invertible.
\end{proof}

\begin{definition}\label{DefwInitial} We define the \emph{associated graded ring} of $A_L$ with respect to $\ma$ to be
\[\Gr_a(A_L):=\bigoplus_{n=0}^{\infty}\ma^n/\ma^{n+1}=A_L/\ma \oplus\,\ma/\ma^2\oplus\ldots.\]
Given $c\in A_L\setminus\{0\}$, the \emph{valuation} $\nu_\ma(c)$ of $c$ with respect to $\ma$ is the unique $n\in\N$ such that $c\in (\ma^n\cdot A_L)\setminus (\ma^{n+1}\cdot A_L)$, and the \emph{initial form} $\initial_\ma(c)$ of $c$ with respect to $\ma$ is its residue class $\overline{c}\in(\ma^{n}\cdot A_L)/(\ma^{n+1}\cdot A_L)\subseteq \Gr_\ma(A_L)$. For an ideal $I\unlhd A_L$, the \emph{initial form} of $I$ with respect to $\ma$ is \[ \initial_\ma(I)=\langle \initial_\ma(c)\mid c\in I \rangle. \]

For our weight vector $w\in \R^{|V|+|Y|}$ and any $g=\sum_{\alpha,\beta}c_{\alpha,\beta} \cdot V^\beta Y^\alpha\in Q^0[Y]$, we define the \emph{initial form} of $g$ with respect to $w$ to be $\initial_w(g)=\sum_{w\cdot (\beta,\alpha)\,\text{max}}\,c_{\alpha,\beta}\cdot V^\beta Y^\alpha \in Q^0[Y]$.
Naturally, for an ideal $J\unlhd Q^0[Y]$, the \emph{initial form} of $J$ with respect to $w$ is \[\initial_w(J)=\langle \initial_w(g)\mid g\in J\rangle.\]
\end{definition}

The map $\initial_w: Q^0[Y] \rightarrow Q^0[Y]$ naturally induces a map $Q^0[Y]_> \rightarrow (Q^0\!/\!J^0)[Y]=K[Y]$, as $\LM_>(p)=1$ implies $\initial_w(p)\in Q^0$ whose image then lies in $K$ and therefore is invertible. We will denote this map with $()_\ini$:
\begin{center}
  \begin{tikzpicture}[description/.style={fill=white,inner sep=2pt}]
    \matrix (m) [matrix of math nodes, row sep=1.5em,column sep=4.5em,text height=1.5ex, text depth=0.25ex]
    { Q^0[Y]_> & K[Y] \\
      Q^0[Y]{}_{\phantom{\ge}} & Q^0[Y] \\ };
    \draw[->] (m-1-1) -- node[above] {$()_\ini$} (m-1-2);
    \draw[->] (m-2-1) -- node[above] {$\initial_w$} (m-2-2);
    \draw[->>] (m-2-2) -- (m-1-2);
    \draw[draw opacity=0] (m-2-1) -- node[sloped] {$\subseteq$} (m-1-1);
  \end{tikzpicture}
\end{center}

Abusing the notation by abbreviating the ideal $\mI\cdot Q^0[Y]$ with $\mI$, we then get:

\begin{lemma}[Lemma 6.6 in \cite{M1991}]\label{lem:graal}
  We have an exact sequence of graded rings
  \begin{center}
    \begin{tikzpicture}[description/.style={fill=white,inner sep=2pt}]
      \matrix (m) [matrix of math nodes, row sep=0.5em,column sep=0em,text height=1.5ex, text depth=0.25ex]
      { 0 &\longrightarrow& \mI_\ini \cdot \, K[Y] &\longrightarrow& K[Y]_{\phantom{d}} &\overset{\lambda}{\longrightarrow}& \Gr_\ma(A_L) &\longrightarrow& 0. \\
          \\
          \\
          & & & & K[Y]_d & \longrightarrow & \ma^d/\ma^{d+1} & \\
          & & & & Y^\alpha & \longmapsto & \overline{f^\alpha} & \\ };
        \draw[draw opacity=0] (m-4-5) -- node[sloped] {$\subseteq$} (m-1-5);
        \draw[draw opacity=0] (m-4-7) -- node[sloped] {$\subseteq$} (m-1-7);
    \end{tikzpicture}
  \end{center}
\end{lemma}
\begin{proof}
It is clear that $\lambda$ is surjective, so remains to show that $\ker(\lambda)=\mI_\ini$.

Suppose $\sum_{|\alpha|=d}\overline{a}_\alpha Y^\alpha \in \ker(\lambda)$ for some $a_\alpha\in A_L$. Then $\sum_{|\alpha|=d}a_\alpha f^\alpha\in\ma^{d+1}$, say $\sum_{|\alpha|=d}a_\alpha f^\alpha=\sum_{|\beta|=d+1}b_\beta f^\beta$ for suitable $b_\beta\in A_L$. Recall the exact sequence in Lemma \ref{lem:Al}. Pick $a_{\alpha,0}, b_{\beta,0}\in Q^0$ such that $\phi(a_{\alpha,0})=a_\alpha$ and $\phi(b_{\beta,0})=b_\beta$ for all $\alpha, \beta$. Then $\sum_{|\alpha|=d}a_{\alpha,0} Y^\alpha-\sum_{|\beta|=d+1}b_{\beta,0} Y^\beta \in \ker(\phi)=\mI$. Thus
\[ \Big(\sum_{|\alpha|=d}a_{\alpha,0} Y^\alpha-\sum_{|\beta|=d+1}b_{\beta,0} Y^\beta\Big)_\ini=\Big(\sum_{|\alpha|=d}a_{\alpha,0} Y^\alpha\Big)_\ini=\sum_{|\alpha|=d}\overline{a}_\alpha Y^\alpha\in\mI_\ini.\]

Conversely, suppose $g_\ini \in\mI_\ini$ for some $g \in \mI$. Writing $g=h+r$, where $h$ is the sum over the terms of lowest degree in $Y$, we have
\[\phi(g)=\underbrace{\phi(h)}_{\in \ma^d}+\underbrace{\phi(r)}_{\in \ma^{d+1}}=0, \text{ where } d=-\deg_w(g) \text{ is the degree in } Y \text{ of g}.\]
This implies that $\lambda(g_\ini)=\lambda(h_\ini)=\overline{0}\in\ma^d/\ma^{d+1}$.
\end{proof}

\begin{proposition}\label{prop:AlGraal}
  The previous two Lemmas yield a diagram
  \begin{center}
    \begin{tikzpicture}[description/.style={fill=white,inner sep=2pt}]
      \matrix (m) [matrix of math nodes, row sep=1.5em,column sep=0em,text height=1.5ex, text depth=0.25ex]
      { 0 &\longrightarrow& \mI \cdot\, Q^0[Y]_> &\longrightarrow& Q^0[Y]_> &\overset{\phi}{\longrightarrow}& A_L &\longrightarrow& 0\phantom{.} \\
        0 &\longrightarrow& \mI_\ini \cdot \, K[Y]{{}_{\phantom{>}}} &\longrightarrow& K[Y]{{}_{\phantom{>}}} &\overset{\lambda}{\longrightarrow}& \Gr_\ma(A_L) &\longrightarrow& 0. \\ };
      \path[->](m-1-3) edge node[right] {$()_\ini$} (m-2-3);
      \path[->](m-1-5) edge node[right] {$()_\ini$} (m-2-5);
      \path[->](m-1-7) edge node[right] {$\initial_\ma()$} (m-2-7);
    \end{tikzpicture}
  \end{center}
  For any $f\in Q^0[Y]_>$ we can calculate the initial form of its image in $A_L$ by
  \[ \initial_\ma(\phi(f)) = \lambda(\NF(f,G_{\mI})_\ini) \quad \text{with} \quad \nu_\ma (\phi(f)) = -\deg_w (\NF(f,G_{\mI})). \]
  In particular, if $\LM_>(f)\notin \LM_>(G_\mI)$, then
  \[ \initial_\ma(\phi(f)) = \lambda(f_\ini) \quad \text{and} \quad \nu_\ma (\phi(f)) = -\deg_w (f). \]
\end{proposition}
\begin{proof}
  The diagram is merely a concatenation of the sequences in Lemma \ref{lem:Al} and \ref{lem:graal}.

  The formula for the initial form and valuation follow immediately from $\ma=\langle \phi(f_1),\linebreak\ldots,\phi(f_k)\rangle$, $\mI=\langle f_1-Y_1,\ldots,f_k-Y_k \rangle$ and our chosen ordering $>$.
\end{proof}

\begin{remark}
Note that if $G$ is a standard basis for $\mI\unlhd Q^0[Y_1,\ldots,Y_s]_>$ with respect to our chosen weighted ordering $>$, then so is $\{\initial_{w}(g) \mid g\in G\}$ for $\initial_{w}(\mI)$. This implies that a standard basis of $\mI$ yields a Gr\"obner basis of $\mI_\ini$.
\end{remark}

\begin{remark}\label{remark}
For practical reasons, it is recommended to compute the minimal polynomial of $K$ over $k(U)$, as we have a natural isomorphism $K\cong Q^0/J^0=k(U)[V]/J^0$. This can be done by applying a generic coordinate transformation
\[ \varphi: Q^0=k(U)[V_1,\ldots,V_n] \rightarrow k(U)[T_1,\ldots,T_{n-1},T]  \]
to push it into general position with respect to a lexicographical ordering (see \cite{GP2008}, Proposition 4.2.2). Depending on the lexicographical ordering, $\varphi(J^0)$ has a Gr\"obner basis of the form $\langle T_1-g_1,\ldots, T_{n-1}-g_{n-1},g\rangle$, where $g_1,\ldots,g_{n-1},g\in k(U)[T]$. We then obtain
\[K \cong Q^0/J^0\cong k(U)[T_1,\ldots,T_{n-1},T]/\varphi(J^0)\cong k(U)[T]/\langle g\rangle, \]
where the last isomorphism is defined by $V_i\mapsto g_i$, for $i=1,\ldots,n-1$.
\end{remark}

\section{Applications}\label{Applications}

In this section we demonstrate how to exploit of the $\ma$-adic topology in several calculations on $A_L$. We give some exemplary questions which allow us to compute in $\Gr_{\ma}(A_L)$ and lift the result back to $A_L$.

\subsection{Dimension}

It is well known that the dimension of a local Noetherian local ring coincides with that of its associated graded ring (see \cite{GP2008} Theorem 5.6.2),
\[ \dim(A_L)=\dim(\Gr_{\ma}(A_L)). \]
Therefore, we can determine the dimension of $A_L$ by computing the dimension of $\Gr_{\ma}(A_L)$.

\subsection{Regularity}\label{regularity}
As we have mentioned in the introduction, localization of affine coordinate rings at the origin is an established tool in the study of singularities. In the nature of these studies, the object of interest are localizations which are not regular.

\begin{definition}\label{regular}
$A_L$ is a regular local ring if \[\dim_{\bK}(\ma/\ma^2)=\dim A_L.\]
\end{definition}

Geometrically, this means that the irreducible subvariety defined by $L$ is not contained in the singular locus of the affine variety given by $A$.
The representation of $\Gr_{\ma}(A_L)$ in Lemma \ref{lem:graal} implies following criterion for the regularity of $A_L$:

\begin{proposition}\label{prop:regularity}
We have
\begin{align*}
  A_L \text{ regular } \quad & \Longleftrightarrow \quad \Gr_{\ma}(A_L) \text{ isomorphic to a polynomial ring } \\
  & \Longleftrightarrow \quad \mI_\ini \text{ generated by linear elements.}
\end{align*}
\end{proposition}
\begin{proof}
Let $H:=\{h_1,\ldots,h_n\}$ be a reduced Gr\"obner basis of $\mI_\ini$, therefore $H$ contains only homogeneous elements.
Let $\dim_{\bK}(\ma/\ma^2)=d$ and without loss of generality suppose $\{Y_1,\ldots,Y_d\}$ is an independent set of variables over $\mI_\ini$. Then $\Gr_{\ma}(A_L)\cong\bK[Y_1,\ldots,Y_d]/\mI_\ini'$, where $\mI_\ini'$ is generated by appropriate linear transformations of all the nonlinear elements of $H$. If $\mI_\ini'=\langle 0\rangle$, it is clear that $\dim(\Gr_{\ma}(A_L))=d$ and thus $A_L$ is regular.

Conversely, suppose $\mI_\ini'\neq\langle 0\rangle$, i.e.~$\Gr_{\ma}(A_L)$ is not a polynomial ring, and suppose without loss of generality that $Y_d$ is a divisor of one of the terms of one of the elements of $\mI_\ini'$.
Then, since  $\bK[Y_1,\ldots,Y_d]/\mI_\ini'$ is a finitely generated $\bK$ algebra, $\dim(\Gr_{\ma}(A_L))=\dim\left(\bK[Y_1,\ldots,Y_d]/\mI_\ini')/\langle Y_1,\ldots, Y_{d-1}\rangle\right)+\textnormal{height} \langle Y_1,\ldots Y_{d-1}\rangle\le0+d-1$. Hence $A_L$ is not regular.
\end{proof}

\begin{example}\label{ex:regularity}
  Consider the affine surface $X$ defined by $y^2+x^3-x^2z^2\in\Q[x,y,z]$. Intersecting $X$ with affine planes $z-t$ yield nodal curves with singular point $(0,0,t)$ for $t\neq 0$, which degenerate into a cuspidal curve for $t=0$.
  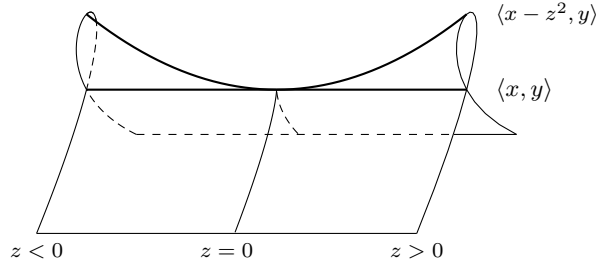
\begin{figure}[h]
    \centering
    \begin{tikzpicture}[x={(0 cm,-2 cm)}, y={(-0.7 cm,-0.7 cm)}, z={(5 cm,0 cm)}, scale=0.5]
      \draw[densely dashed, domain = -1.5:-1] plot ({(\x)^2-1},{\x^3 - \x},-1);
      \draw[domain = -1:0.5] plot ({(\x)^2-1},{\x^3 - \x},-1);
      \draw[densely dashed, domain = 0:1] plot ({(\x)^2-1},{\x^3 - \x},-1);
      \draw[domain = 1:1.5] plot ({(\x)^2-1},{\x^3 - \x},-1);

      \draw[densely dashed, domain = -0.93:0] plot ({(\x)^2},{\x^3},0);
      \draw[domain = 0:1.16] plot ({(\x)^2},{\x^3},0);

      \draw[domain = -1.5:1.5] plot ({(\x)^2-1},{\x^3 - \x},1);

      \draw (1.25,1.875,-1) -- (1.25,1.875,1);

      \draw[dashed, shorten >= 0.825cm] (1.25,-1.875,-1) -- (1.25,-1.875,1);
      \draw[shorten <= 4.175cm] (1.25,-1.875,-1) -- (1.25,-1.875,1);

      \draw[thick] (0,0,-1) -- (0,0,1);
      \draw[thick, domain = 0:1] plot (-\x^2,0,\x);
      \draw[thick, domain = -1:0] plot (\x^2,0,\x);
      \node[anchor=west, font=\footnotesize] at (0,0,1.1) {$\langle x,y\rangle$};
      \node[anchor=west, font=\scriptsize] at (-1,0,1.1) {$\langle x-z^2,y\rangle $};

      \node[anchor=north, font=\scriptsize] at (1.25,1.875,-1) {$z<0$};
      \node[anchor=north, font=\scriptsize] at (1.25,1.875,0) {$z=0$};
      \node[anchor=north, font=\scriptsize] at (1.25,1.875,1) {$z>0$};
    \end{tikzpicture}
    \caption{$V(y^2+x^3-x^2z^2)$}
    \label{fig:systemOfParameters}
  \end{figure}

  Consequently, the singular locus of $X$ is the line given by
  \[ L:=\langle x, y\rangle\unlhd A:=\Q[x,y,z] / \langle y^2+x^3-x^2z^2\ \rangle. \]
  Indeed, setting $f_1:= x$ and $f_2:= y$ we see that
  \[ z^2 \cdot f_1^2+ f_2^2 = f_1^3\in A\subseteq A_L, \]
  which means in the isomorphism induced by Lemma \ref{lem:graal} that $z^2\cdot Y_1^2+Y_2^2=0\in\Gr_\ma(A_L)$. By Proposition \ref{prop:regularity}, $A_L$ is then not regular.

  On the other hand for
  \[ L':=\langle x-z^2, y\rangle \unlhd A:=\Q[x,y,z] / \langle y^2+x^3-x^2z^2\ \rangle.  \]
  we can compute that $\Gr_\ma(A_L) = K[Y_1,Y_2]$,
  where $K=\Q[T]/T^2$ and $Y_1,Y_2$ represent the generators $x-z^2, y$. Therefore, $A_{L'}$ is regular.
\end{example}

\subsection{System of parameters}

A system of parameters can be thought of as a local coordinate system, which determines the points around the subvariety up to finite ambiguity.

\begin{definition}
Let $d:=\dim(A_L)$. Then $\{a_1,\ldots,a_d\}$ is called a system of parameters of $A_L$, if $\langle a_1,\ldots, a_d\rangle$ is $\ma$-primary. 
If $\langle a_1,\ldots, a_d\rangle=\ma$ then it is called a regular system of parameters.
\end{definition}

Note that if $a_1,\ldots, a_d\in A_L$ is such that the radical of $\langle\initial_{\ma}(a_1),\ldots,\initial_{\ma}(a_d)\rangle\unlhd \Gr_{\ma}(A_L)$ is $\langle Y_1,\ldots, Y_s\rangle/\mI_\ini$, then for $I:={\langle a_1,\ldots, a_d\rangle}$, we have $(\sqrt{I}\cap\ma^i)/\ma^{i+1}=\ma^i/\ma^{i+1}$ for all $i$. But, since $\cap_i\ma^i=0$, this implies that $\sqrt{I}=\ma$.

Hence, knowing that an ideal in $A_L$ or $\Gr_{\ma}(A_L)$ is primary if its radical is maximal, finding a system of parameters for $A_L$ boils down to finding a set of homogeneous elements $\{\lambda_1,\ldots,\lambda_d\}\subset\bK[Y_1,\ldots, Y_s]$ such that the radical of $\bH=\langle \lambda_1,\ldots,\lambda_d\rangle +\mI_\ini$ is $\langle Y_1,\ldots, Y_s\rangle$.

Now, it follows from linear algebra that for a generic choice of $c_{ij}\in Q^0$, $\lambda_{i}=\sum c_{ij}Y_j$ satisfies the above condition (which can be tested by computing the dimension of $\bH$). Choosing $a_i$ as $\sum c_{ij}{f}_i$, where $\langle {f}_1,\ldots, f_s\rangle=L\subseteq A\subseteq A_L$, we have that $\initial_{\ma}(a_i)=\lambda_i$ and we are done.

A regular system of parameters only exists when $A_L$ is regular, which can be checked using the method described in Section \ref{regularity}. In such a case, a regular system of parameters can be determined in the same way, except checking whether $\bH$ equals $\langle Y_1,\ldots, Y_s\rangle$ instead of checking whether the radical of $\bH$ equals $\langle Y_1,\ldots, Y_s\rangle$.

\begin{example}
  For $A_L$ as in Example \ref{ex:regularity}, a simple possible system of parameters is $\{y\}$. It means that if we look alongside the line given by $L\unlhd A$, the function $y\in A$ distinguishes the points of $A$ in the near vicinity of $L$ up to finite ambiguity as in Figure \ref{fig:systemOfParameters}.
  \begin{figure}[h]
    \centering
    \begin{tikzpicture}
      \draw[densely dotted, domain = -1.2:1.2] plot ({(\x)^2},{\x^3});
      \draw[dashed, domain = -1.5:1.5] plot ({(\x)^2-1},{\x^3 - \x});
      \draw (-0.5,2) -- (-0.5,-2);
      \fill (0,0) circle (2pt);
      \node[anchor=west] (L) at (1,0) {$V(L)$};
      \draw[->,shorten >= 0.3cm] (L) -- (0,0);
      \node[anchor=west] (A) at (1.5,1.75) {$X(A)$};
      \node[anchor=north] (y) at (-0.5,-2) {$y=\varepsilon$};
      \node[anchor=west] at (y.east) {for some $\varepsilon<0$};
    \end{tikzpicture}
    \caption{$V(y^2+x^3-x^2z^2)$ around $V(x,y)$}
    \label{fig:systemOfParameters}
  \end{figure}
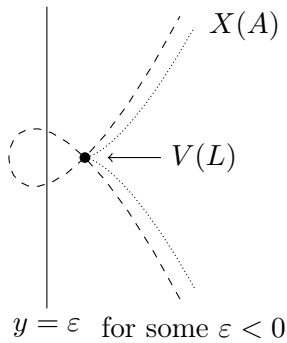
\end{example}

\subsection{Hilbert-Samuel function and Hilbert-Samuel polynomial}
For a homogeneous ideal $I\unlhd \Gr_\ma(A_L)$, i.e.~an ideal with an homogeneous preimage, we denote the Hilbert function of $I$ by
\[\Hilb_{\Gr_\ma(A_L)}(I,n)=\dim_K (\Gr_\ma(A_L)/I)_n.\]
Since the the Hilbert-Samuel function of an ideal $J\unlhd A_L$, denoted by $\HilbS_{A_L}(J)$, evaluated at $n$ equals $\dim(J/(\ma^n\cdot J))$ by definition, it follows that \[\HilbS_{A_L}(J,n)=\sum_{i=1}^n\Hilb_{\Gr_\ma{A_L}}(\initial_\ma(J),i-1).\]
Suppose the Hilbert polynomial of $\initial_\ma(J)\unlhd\Gr_\ma(A_L)$ is of the form
\begin{equation}\label{Hilb}
\HilbP_{\Gr_\ma(A_L)}(\initial_\ma(J),x)=\sum_{\nu=0}^{s-1}a_\nu\binom{x}{\nu},
\end{equation}
for $a_\nu\in\Z$, so that $\HilbP_{\Gr_\ma(A_L)}(\initial_\ma(J),n)=\Hilb_{\Gr_\ma(A_L)}(\initial_\ma(J),n)$ for $n$ sufficiently large. Then the Hilbert Samuel polynomial is given by
\begin{equation}\label{HilbS}
\HilbSP_{A_L}(J,x):=\sum_{\nu=1}^{s}a_{\nu-1}\binom{x}{\nu}+c,
\end{equation}
where $c=\dim(J/(\ma^l\cdot J))-\sum_{\nu=1}^{s}a_{\nu-1}\binom{l}{\nu}$, for any sufficiently large $l$, so that $\HilbSP_{A_L}(J,n)=\HilbS_{A_L}(J,n)$. And, by Corollaries 5.1.5 and 5.5.5 in \cite{GP2008}, any $l\ge d$ is sufficiently large, where $d$ is the degree of the second Hilbert series of $\initial_\ma(J)$. 

Taking the relationship between (\ref{Hilb}) and (\ref{HilbS}) into account it suffices to compute $\HilbP_{\Gr_\ma(A_L)}(\initial_\ma(J))$ to determine $\HilbSP_{A_L}(J)$. Since $\Gr_\ma(A_L)/\initial_\ma(J)\cong K[Y]/(\mI_\ini+\initial_\ma(J))$ it follows that $\Hilb_{\Gr_\ma(A_L)}(J_\ini )=\Hilb_{K[Y]}(\mI_\ini+\initial_\ma J)=\Hilb_{K[Y]}(L(\mI_\ini+\initial_\ma J))$, where $L(I)$, denotes the leading ideal of $I$.

\begin{remark}
  This algorithm is still subject to ongoing work, because algebraic field extensions over transcendental field extensions are not available in Singualr yet.
\end{remark}

\subsection{Syzygies and Resolutions}\
For this section, fix an ideal $I\unlhd A_L$ and consider an $A_L$-free resolution
\[ 0 \longleftarrow A_L/I \longleftarrow A_L \longleftarrow M_1 \longleftarrow M_2 \longleftarrow \cdots, \]
where, for sake of simplicity, all $M_i$ are free $A_L$ modules of finite rank, and all maps are compatible with the filtrations. That means for example that the first two maps are of the form
\begin{center}
  \begin{tikzpicture}[description/.style={fill=white,inner sep=2pt}]
    \matrix (m) [matrix of math nodes, row sep=0.5em, column sep=0em, text height=1.5ex, text depth=0.25ex]
    { g_{i,0} & \longmapsfrom  & e_{i,1} \\
      A_L     & \longleftarrow & M_1     \\
      \\
              &                & M_1     & \longleftarrow & M_2     \\
              &                & g_{i,1} & \longmapsfrom  & e_{i,2} \\ };
    \node[anchor=west] at (m-2-3.east) {$=\, \bigoplus_{i=1}^{k_1}A_L(-\nu_\ma(g_{i,0})),$};
    \node[anchor=west] at (m-4-5.east) {$=\, \bigoplus_{i=1}^{k_2}A_L(-\ord(g_{i,1})),$};
  \end{tikzpicture}
\end{center}
where $e_{i,1}$, $e_{i,2}$ refer to the canonical basis elements of $M_1$ and $M_2$, and $g_{i,0}$, $g_{i,1}$ refer to their images respectively. Moreover, $\ord$ denotes the order function of the filtration on $M_1$, which has been twisted so that $\ord(1\cdot e_{i,1})=0+\nu_\ma(g_{i,0})$, hence the map takes the $n$-th filtration module of $M_1$ to the $n$-th filtration module of $A_L$ for any $n\in\N$.

\begin{definition}
  Let $M$ be an $A_L$-module as above, say $M=\oplus_{i=1}^k A_L(-d_i)$. Then its \emph{associated graded module} is the free graded $\Gr_\ma(A_L)$-module given by
  \[ \Gr_\ma(M) := \bigoplus_{i=1}^k \Gr_\ma(A_L)(-d_i), \]
  so that for any $d\in\N$ its degree $d$ component is given by
  \[ \Gr_\ma(M)_d := \bigoplus_{i=1}^k \Gr_\ma(A_L)_{d+d_i} = \bigoplus_{i=1}^k \ma^{d+d_i}/ \ma^{d+d_i+1}. \]

  For any $g=\sum_{i=1}^k g_i\cdot e_i \in M$, say $\ord(g)=n$, the \emph{initial form} of $g$ is defined to be
  \[ \initial_\ma(g):=\sum_{i=1}^k \overline{g}_i\cdot e_i \in \Gr_\ma(M)_{n}\subseteq \Gr_\ma(M), \]
  so that $\overline{g}_i\cdot e_i = 0 \in \Gr_\ma(M)_{n}$ for all $\ord(g_i\cdot e_i)>\ord(g)$.
  The initial module of a submodule $M'\subseteq M$ is naturally defined to be $\initial(M'):=\langle \initial(g)\mid g\in M'\rangle$.
\end{definition}

It is straightforward to see how our $A_L$-free resolution of $I$ yields a graded, $\Gr_\ma(A_L)$-free resolution of $\ini_\ma(I)$ due to the functorial nature of taking initial forms:
\begin{center}
\begin{tikzpicture}[description/.style={fill=white,inner sep=2pt}]
\matrix (m) [matrix of math nodes, row sep=1.5em,column sep=0em,%
text height=1.5ex, text depth=0.25ex]
{ 0 & \longleftarrow & A_L/I & \longleftarrow & A_L & \longleftarrow & M_1 & \longleftarrow & M_2 & \longleftarrow & \cdots\phantom{.} \\
  0 & \longleftarrow & \Gr_\ma(A_L)/\initial_\ma(I) & \longleftarrow & \Gr_\ma(A_L) & \longleftarrow & N_1 & \longleftarrow & N_2 & \longleftarrow & \cdots. \\ };
\path[->](m-1-3) edge node[auto] {$\ini_\ma$}(m-2-3);
\path[->](m-1-5) edge node[auto] {$\ini_\ma$}(m-2-5);
\path[->](m-1-7) edge node[auto] {$\ini_\ma$}(m-2-7);
\path[->](m-1-9) edge node[auto] {$\ini_\ma$}(m-2-9);
\end{tikzpicture}
\end{center}

Given the notation from the very beginning, the first two maps of our resolution would be of the form
\begin{center}
  \begin{tikzpicture}[description/.style={fill=white,inner sep=2pt}]
    \matrix (m) [matrix of math nodes, row sep=0.5em, column sep=0em, text height=1.5ex, text depth=0.25ex]
    { \initial_\ma(g_{i,0}) & \longmapsfrom  & e_{i,1} \\
      \Gr_\ma(A_L)          & \longleftarrow & N_1     \\
      \\
              &                & N_1     & \longleftarrow & N_2     \\
              &                & \initial_\ma(g_{i,1}) & \longmapsfrom  & e_{i,2} \\ };
    \node[anchor=west] at (m-2-3.east) {$=\, \bigoplus_{i=1}^{k_1}\Gr_\ma(A_L)(-\nu_\ma(g_{i,0})),$};
    \node[anchor=west] at (m-4-5.east) {$=\, \bigoplus_{i=1}^{k_2}\Gr_\ma(A_L)(-\ord(g_{i,1})),$};
  \end{tikzpicture}
\end{center}
where, for sake of simplicity, we abuse notation and use $e_{i,1}$ respectively $e_{i,2}$ to refer to the canonical basis elements of both $M_1$ and $N_1$ respectively $M_2$ and $N_2$.

The goal of this section is to establish an inverse process, in which we will lift a graded $\Gr_\ma(A_L)$-free resolution of $\initial_\ma(I)$ to an $A_L$-free resolution of $I$. For that, fix a graded $\Gr_\ma(A_L)$-free resolution for the remainder of the chapter
\[ 0 \longleftarrow \Gr_\ma(I)/\initial_\ma(I) \longleftarrow \Gr_\ma(A_L) \longleftarrow N_1 \longleftarrow N_2 \longleftarrow \cdots. \]

Roughly speaking, the lift of the whole resolution consists of repeated lifts of syzygies over $\Gr_\ma(A_L)$ to syzygies over $A_L$, i.e. for any free $\Gr_\ma(A_L)$-module $N=N_j$ in our resolution and two sets $\Theta=\{\theta_1,\ldots,\theta_k\}\subseteq N$ and $\Delta=\{\delta_1,\ldots,\delta_k\}\subseteq M:=M_j$ with $\initial_\ma(\delta_i)=\theta_i$, compute for a given $\eta\in \syz(\Theta)$ a $\gamma\in \syz(\Delta)$ with $\initial_\ma(\gamma)=\eta$. In our setting, both $\Theta$ and $\Delta$ will typically be systems of generators of previous syzygy modules, while $\eta$ is the image of a canonical basis element in the resolution over $\Gr_\ma(A_L)$ and $\gamma$ will then be the image of a canonical basis element in the resolution over $A_L$.

\begin{center}
\begin{tikzpicture}[description/.style={fill=white,inner sep=2pt}]
  \matrix (m) [matrix of math nodes, row sep=1.5em, column sep=0em, text height=1.5ex, text depth=0.25ex]
  { \ldots & \longleftarrow & M & \longleftarrow & \bigoplus_{i=1}^k A_L(-\ord(\delta_i))  \\
    \ldots & \longleftarrow & N & \longleftarrow & \bigoplus_{i=1}^k \Gr_\ma(A_L)(-\deg(\theta_i)) & \longleftarrow & \ldots \\ };
  \node[anchor=south,yshift=0.35cm] (delta) at (m-1-3.north) {$\Delta$};
  \node[anchor=north,yshift=-0.35cm] (theta) at (m-2-3.south) {$\Theta$};
  \draw[draw opacity=0] (delta) -- node[sloped] {$\subseteq$} (m-1-3);
  \draw[draw opacity=0] (theta) -- node[sloped] {$\subseteq$} (m-2-3);

  \path[->](m-1-3) edge node[auto, font=\scriptsize] {$\ini_\ma$}(m-2-3);
  \path[->](m-1-5) edge node[auto, font=\scriptsize] {$\ini_\ma$}(m-2-5);

  \node[anchor=south,yshift=0.35cm] (gamma) at (m-1-5.north) {$\gamma$};
  \node[anchor=north,yshift=-0.35cm] (eta) at (m-2-5.south) {$\eta$};
  \draw[draw opacity=0] (gamma) -- node[sloped] {$\in$} (m-1-5);
  \draw[draw opacity=0] (eta) -- node[sloped] {$\in$} (m-2-5);

  \node[anchor=north,yshift=-0.35cm] (ei) at (m-2-7.south) {$e_i$};
  \draw[draw opacity=0] (ei) -- node[xshift=0.9cm] {$\longmapsfrom$} (eta);

  \node[anchor=east, xshift=0.1cm, yshift=0.03cm] at (gamma.west) {\phantom{$\gamma$}find};
  \node[anchor=west, xshift=-0.1cm, yshift=0.03cm] at (gamma.east) {such that $\initial_\ma(\gamma)=\eta$};

  \node[anchor=west, yshift=1.2cm, font=\scriptsize] at (m-2-6.west) {$\bigg(\begin{tabular}{c}so that here \\ $\gamma \mapsfrom e_i$\end{tabular}\bigg)$};
\end{tikzpicture}
\end{center}

Suppose $N=\oplus_i \Gr_\ma(A_L)(-d_i)$ and $M:=\oplus_i A_L(-d_i)$. Then Lemma \ref{lem:graal} immediately induces a surjective morphism $\lambda$ from the free $K[Y]$-module $\widehat N=\oplus_i K[Y](-d_i)$ onto $N$ with kernel $\ker(\lambda)=\mI_\ini\cdot \widehat N$ and Lemma \ref{lem:Al} induces a surjective morphism $\phi$ from a free $Q^0[Y]_>$-module $\widehat M=\oplus_i Q^0[Y]_>(-d_i)$ onto $M$ with kernel $\ker(\phi)=\mI\cdot \widehat M$. Note that $\widehat N$ and $\widehat M$ are shifted similarly to their images so that $\lambda$ remains a graded map of degree $0$ and $\phi$ remains compatible with the filtrations.

Moreover, if we consider a weight on the module variables and an ordering on the module monomials of $\widehat M$ that are compatible with the filtration, Proposition \ref{prop:AlGraal} immediately carries over to our module setting. For that, we extend our existing weight vector $w$ to weight the canonical basis elements $e_i$ with weight $d_i$. Note that our lifted resolution will induce a sequence on the free $Q^0[Y]_>$-modules $\widehat M$. Extend $>$ to the Schreyer ordering on the module monomials on $\widehat M$.

Note that a Gr\"obner basis of $\ker(\lambda)=\mI_\ini\cdot \widehat N$ and a standard basis $\ker(\phi)=\mI\cdot \widehat M$ are not hard to obtain, as they can be easily derived from the Gr\"obner basis of $\mI_\ini$ and the standard basis of $\mI$ respectively.

\begin{center}
  \begin{tikzpicture}[description/.style={fill=white,inner sep=2pt}]
    \matrix (m) [matrix of math nodes, row sep=1.75em,column sep=1em,text height=1.5ex, text depth=0.25ex]
    { 0 &\longrightarrow& \mI \cdot \widehat M &\longrightarrow& \widehat M &\overset{\phi}{\longrightarrow}& M &\longrightarrow& 0\phantom{.} \\
      0 &\longrightarrow& \mI_\ini \cdot \widehat N &\longrightarrow& \widehat N &\overset{\lambda}{\longrightarrow}& N &\longrightarrow& 0. \\ };
    \path[->, shorten <= 3pt, shorten >= 3pt, font=\scriptsize] (m-1-3) edge node[right] {$()_\ini$} (m-2-3);
    \path[->, shorten <= 3pt, shorten >= 3pt, font=\scriptsize] (m-1-5) edge node[right] {$()_\ini$} (m-2-5);
    \path[->, shorten <= 3pt, shorten >= 3pt, font=\scriptsize] (m-1-7) edge node[right] {$\initial_\ma()$} (m-2-7);
    \node[anchor=south,yshift=0.5cm, xshift=0.5cm] (delta) at (m-1-7.north) {$\bigoplus_i A_L(-d_i)$};
    \node[anchor=north,yshift=-0.35cm, xshift=0.5cm] (theta) at (m-2-7.south) {$\bigoplus_i \Gr_\ma(A_L)(-d_i)$};
    \draw[draw opacity=0] (delta) -- node[sloped] {$=$} (m-1-7);
    \draw[draw opacity=0] (theta) -- node[sloped] {$=$} (m-2-7);
    \node[anchor=north,yshift=-0.35cm, xshift=-0.5cm] (H) at (m-2-5.south) {$\bigoplus_i K[Y](-d_i)$};
    \draw[draw opacity=0] (H) -- node[sloped] {$=$} (m-2-5);
    \node[anchor=south,yshift=0.5cm, xshift=-0.5cm] (G) at (m-1-5.north) {$\bigoplus_i Q^0[Y]_>(-d_i)$};
    \draw[draw opacity=0] (G) -- node[sloped] {$=$} (m-1-5);
  \end{tikzpicture}
\end{center}

Now let $\{\delta_1,\ldots,\delta_k\}$ be a generating set of a submodule $M'\leq M$ such that $\{\theta_1,\ldots,\theta_k\}$, where $\theta_i:=\initial_\ma(\delta_i)$, generates $N':=\initial_\ma(M')\leq N$. To lift syzygies of $(\theta_1,\ldots,\theta_k)$ to syzygies of $(\delta_1,\ldots,\delta_k)$ will require a standard basis of $\phi^{-1}(M')\leq\widehat M$ as well as a representation of the basis elements via suitable preimages of $\delta_i$.

Since our utmost goal is to avoid complicated computations over $A_L$, it stands out of question to compute this standard basis of $\phi^{-1}(M')$ from scratch. However, this is not needed since Gr\"obner bases of $\phi^{-1}N'$ can be naturally lifted to standard bases of $\phi^{-1}(M')$.

\begin{center}
  \begin{tikzpicture}[description/.style={fill=white,inner sep=2pt}]
    \matrix (m) [matrix of math nodes, row sep=1.75em,column sep=1em,text height=1.5ex, text depth=0.25ex]
    { 0 &\longrightarrow& \mI \cdot \widehat M &\longrightarrow& \widehat M &\overset{\phi}{\longrightarrow}& M &\longrightarrow& 0\phantom{.} \\
      0 &\longrightarrow& \mI_\ini \cdot \widehat N &\longrightarrow& \widehat N &\overset{\lambda}{\longrightarrow}& N &\longrightarrow& 0. \\ };
    \path[->, shorten <= 3pt, shorten >= 3pt, font=\scriptsize] (m-1-3) edge node[right] {$()_\ini$} (m-2-3);
    \path[->, shorten <= 3pt, shorten >= 3pt, font=\scriptsize] (m-1-5) edge node[right] {$()_\ini$} (m-2-5);
    \path[->, shorten <= 3pt, shorten >= 3pt, font=\scriptsize] (m-1-7) edge node[right] {$\initial_\ma()$} (m-2-7);
    \node[anchor=south,yshift=0.5cm] (Mprime) at (m-1-7.north) {$M'$};
    \node[anchor=base west,xshift=-0.15cm] at (Mprime.base east) {$\supseteq \{ \delta_1,\ldots,\delta_k \}$};
    \node[anchor=north,yshift=-0.35cm] (Nprime) at (m-2-7.south) {$N'$};
    \node[anchor=base west,xshift=-0.15cm] at (Nprime.base east) {$\supseteq \{ \theta_1,\ldots,\theta_k \}$};
    \draw[draw opacity=0] (Mprime) -- node[sloped] {$\leq$} (m-1-7);
    \draw[draw opacity=0] (Nprime) -- node[sloped] {$\leq$} (m-2-7);
    \node[anchor=north,yshift=-0.35cm] (H) at (m-2-5.south) {$\lambda^{-1}N'$};
    \node[anchor=base east,xshift=0.15cm] at (H.base west) {$\{ h_1,\ldots,h_k \}\subseteq$};
    \node[anchor=south,yshift=0.425cm] (G) at (m-1-5.north) {$\phi^{-1}M'$};
    \node[anchor=base east,xshift=0.15cm] at (G.base west) {$\{ g_1,\ldots,g_k \}\subseteq$};
    \draw[draw opacity=0] (H) -- node[sloped] {$\leq$} (m-2-5);
    \draw[draw opacity=0] (G) -- node[sloped] {$\leq$} (m-1-5);
    \node[anchor=north,yshift=-0.5cm] (HDash) at (H.south) {$\{h_1',\ldots,h_l',h_{l+1}',\ldots\}$};
    \draw[draw opacity=0] (HDash) -- node[sloped] {$\subseteq$} (H);
    \node[anchor=base west] (HDashExpl0) at (HDash.base east) {standard basis with};
    \node[anchor=base west, yshift=-0.65cm] (HDashExpl1) at (HDash.base east) {$h_j'=\sum_{i=1}^k q_{ji} \cdot h_i + s_j$};
  \end{tikzpicture}
\end{center}

For that, consider preimages $g_1,\ldots,g_k$ of $\delta_1,\ldots,\delta_k$ with $\LM_>(g_i)\notin \LM_>(\mI\cdot\widehat M)$, so that $h_1,\ldots,h_k$, where $h_i:=g_{i,\ini}$, are homogeneous preimages of $\theta_1,\ldots,\theta_k$ with $\LM_>(h_i)\notin\LM_>(\mI_\ini\cdot\widehat N)$. Consider a homogeneous Gr\"obner basis $\{h_1',\ldots,h_l',h_{l+1}',\ldots\}$ of $\lambda^{-1}N'$ with $\LM_>(h_i')\notin\LM_>(\mI_\ini\cdot\widehat N)$ for $1\leq i\leq l$ and $\LM_>(h_i')\in\LM_>(\mI_\ini\cdot\widehat N)$ otherwise. Then there exists $q_{ji}\in Q^0[Y]$ such that
  \[ h_j'=\sum_{i=1}^k q_{ji} \cdot h_i + s_j \text{ for some } s_i\in \mI_\ini \text{ for all } i=1,\ldots,l. \]
Since all $h_j'$ and $h_i$ are homogeneous (with variables $Y$), we may assume that all $q_{ji}$ are weighted homogeneous with $-\deg_w(q_{ji})=\deg(h_i')-\deg(h_j)$, unless $q_{ji}=0$. Moreover, we may assume that $\LM_>(q_{ji})\notin \mI$ unless $q_{ji}=0$, as this can always be achieved by modifying $s_i$, which in particular implies that $\nu_\ma(\phi(q_{ji}))=-\deg_w(q_{ji})$.

With this, we have everything necessary to formulate our algorithm for lifting Gr\"obner bases.

\begin{algorithm}[H]
\caption{lifting Gr\"obner bases}
\label{liftingGB}
\begin{algorithmic}[1]
\Require{Given two submodules $M'\leq M$ and $N'\leq N$ with $\initial_\ma(M')=N'$,
  \begin{enumerate}[leftmargin=*]
  \item a standard basis $G_0$ of $\mI\cdot \widehat M$,
  \item $G=\{g_1,\ldots,g_k\}\subseteq \widehat M$ with $\LM_>(g_i)\notin\LM_>(\mI\cdot\widehat M)$, 
  \item $\{h_1',\ldots,h_l',h_{l+1}',\ldots\}$ a Gr\"obner basis of $\lambda^{-1}N'$ with $\LM_>(h_i')\notin\LM_>(\mI_\ini\cdot\widehat N)$ for $1\leq i\leq l$ and $\LM_>(h_i')\in\LM_>(\mI_\ini\cdot\widehat N)$ otherwise,
  \end{enumerate}}
\algstore{liftingGroebnerBases}
\end{algorithmic}
\end{algorithm}

\begin{algorithm}[H]
\begin{algorithmic}[1]
\algrestore{liftingGroebnerBases}
\Require{(continued)
  \begin{enumerate}[leftmargin=*]
  \setcounter{enumi}{3}
  \item $q_{ji}\in Q^0[Y]$ homogeneous with respect to weight vector $w$ and $\LM_>(q_{ji})\notin \mI$ unless $q_{ji}=0$ such that
    \[ h_j'=\sum_{i=1}^k q_{ji} \cdot h_i + r_j \text{ for all } j=1,\ldots,l, \]
    where $h_i:=g_{i,\ini}$ for $i=1,\ldots,k$.
  \end{enumerate}}
\Ensure{$G'=\{g_1',\ldots,g_l'\}\subseteq \widehat M$ and $q_{ji}'\in Q^0[Y]$ such that
  \begin{enumerate}
  \item $G' \cup G_0$ is a standard basis of $\phi^{-1}M'$,
  \item $g_j'=\sum_{i=1}^k q_{ji}'\cdot g_i + r_j$ for some $r_j\in \mI$,
  \item $\ord( \phi(g_j')) =\ord (\phi(q_{ji}'\cdot g_i))$ unless $q_{ji}'\cdot g_i=0$.
  \end{enumerate}}
\For{$j=1,\ldots,l$}
  \State Compute
    \[ g_j':= \NF\Big(\sum_{i=1}^k q_{ji} \cdot g_i,G_0\Big)\in Q^0[Y]. \]
  \If{$\LM_>(g_j')=s_j\cdot Y^\alpha$ for some $s_j\in Q^0$, $s_j\neq 1$}
    \State Find $a_j \in Q^0$ such that $\overline{a}_j\cdot \overline{s}_j=1 \in K=Q^0/J^0$.
    \State Redefine
      \[ g_j':= \NF(a_j \cdot g_j', G_0). \]
  \Else
    \State Set $a_j:=1\in Q^0$.
  \EndIf
\EndFor
\Return{$\{g_1',\ldots,g_l'\}$ and $(a_i\cdot q_{ji})$.}
\end{algorithmic}
\end{algorithm}
\begin{proof}
  In order to show that $G'\cup G_0$ is a standard basis of $\phi^{-1}M'$, we need show that $\LM_>(g_{j,\ini}')=\LM_>(h_j')$ for all $j=1,\ldots,l$ first. For sake of clarity, rename $g_j'$ as it appears in Steps 2 to 5 to $g_j''$, and set
  \begin{align*}
   r_{j,1} &:= \sum_{i=1}^k q_{ji} \cdot g_i - \NF\Big(\sum_{i=1}^k q_{ji} \cdot g_i,G_0\Big) \text{ as in Step 2}, \\
   r_{j,2} &:= \begin{cases} a_j \cdot g_j'' - \NF(a_j \cdot g_j'', G_0) &\text{ as in Step 5 if } s_j\neq 1, \\ 0 &\text{ otherwise,} \end{cases}
  \end{align*}
  so that $g_j'' = \sum_{i=1}^k q_{ji} \cdot g_i - r_{j,1}$ and $g_j' = a_j\cdot g_j''-r_{j,2}$, or rather
  \[ g_j' = a_j\cdot \sum_{i=1}^k q_{ji} \cdot g_i - a_j\cdot r_{j,1} - r_{j,2}. \]
  By definition of $>$ and unless $r_{j,1}=0$ or $r_{j,2}=0$, we have
  \begin{align*}
    \deg_w(r_{j,1})&\leq \deg_w\Big(\sum_{i=1}^k q_{ji} \cdot g_i\Big), \\
    \deg_w(r_{j,2})&\leq \deg_w(a_j \cdot g_j'') \leq  \deg_w\Big(a_j \cdot \sum_{i=1}^k q_{ji}\cdot g_i\Big).
  \end{align*}
  Setting $c_1$ and $c_2$ to be either $1$ or $0$ depending on the disparity of weighted degree, together, this yields
  \begin{equation}
    \label{eq:gDashIn}
    g_{j,\ini}' = \underbrace{\Big(a_j\cdot \sum_{i=1}^k q_{ji} \cdot g_i\Big)_\ini}_{=a_j\cdot h_j' \text{ by } (4)} - c_1\cdot \underbrace{(a_j\cdot r_{j,1})_\ini}_{\in\mI_\ini} -\; c_2\cdot \underbrace{(r_{j,2})_\ini}_{\in\mI_\ini}. \tag{*}
  \end{equation}
  Because $\LM_>(g_j')\notin \LM_>(\mI\cdot\widehat M)$ due to the normal form computations, we hence must have $\LM_>(g_{j,\ini}')\notin \LM_>(\mI_\ini\cdot\widehat N)$, which leaves $\LM_>(g_{j,\ini}')=\LM_>(a_j\cdot h_j')=\LM_>(h_j')$ as the only possibility.

  To show that $G'\cup G_0$ is a standard basis of $\phi^{-1}M'$, consider an element $f\in \phi^{-1}M'$.
  If $\LM_>(f)\in \LM_>(\mI \cdot \widehat M)$, then by assumption (1) there exists an element of $G_0$ with leading monomial dividing $\LM_>(f)$.

  If $\LM_>(f)\notin \LM_>(\mI \cdot \widehat M)$, then in particular $\LM_>(f_\ini)\notin \LM_>(\mI_\ini\cdot \widehat N)$, and by assumption (3) there exists an $j=1,\ldots,l$ such that $\LM_>(h_j')\mid \LM_>(f_\ini)$.
  Suppose $s_i\neq 1$ during Step $3$ in the $j$-th iteration of the \texttt{for} loop. In Step $4$ we then find an $a_j\in Q^0$ such that $a_j\cdot s_j=r+1$ for some $r\in J^0$. By construction of $\mI$ and the choice of our monomial ordering $>$, we therefore have $\LM_>(g_i')=1\cdot Y^\alpha$ after the normal form computation in Step $5$.
  Either way, before the end of our \texttt{for} loop iteration in Step $6$, we have
  \[ \LM_>(g_i')=Y^\alpha =\LM_>(g_{i,\ini}')=\LM_>(h_i') \mid \LM_>(f_\ini). \]
  And since $\LM_>(f_\ini) = \LM_>(f)\mid_{V=1}$ it follows that $\LM_>(g_i')\mid \LM_>(f)$.

  Moreover, our previous considerations imply for all $j=1,\ldots,l$
  \[ g_j' = a_j\cdot \sum_{i=1}^k q_{ji} \cdot g_i - a_j\cdot r_{j,1} - r_{j,2} = \sum_{i=1}^k q_{ji}' \cdot g_i - \underbrace{a_j\cdot r_{j,1} - r_{j,2}}_{\in \mI\cdot \widehat M}, \]
  showing the second condition of our output.

  For the last condition note that we have
  \begin{center}
    \begin{tikzpicture}[description/.style={fill=white,inner sep=2pt}]
      \node (11) at (0,0) {$\ord(\phi(g_j'))$};
      \node[anchor=base west, xshift=-0.2cm] (110) at (11.base east) {$=$};
      \node[anchor=base west, xshift=-0.2cm] (12) at (110.base east) {$\deg(\lambda(a_j\cdot h_j'))$};
      \node[anchor=north, yshift=-0.5cm] (14) at (12.south) {$-\deg_w(h_i')$};
      \draw[draw opacity=0] (14) -- node[sloped] {$=$} (12);
      \node[anchor=base west, xshift=-0.2cm] (140) at (14.base east) {$=$};
      \node[anchor=base west, xshift=-0.2cm] (15) at (140.base east) {$-\deg_w(q_{ji})-\deg_w(h_i)$};
      \node[anchor=base west, xshift=-0.2cm] (150) at (15.base east) {$=$};
      \node[anchor=base west, xshift=-0.2cm] (16) at (150.base east) {$-\deg_w(q_{ji})-\deg_w(g_i)$};
      \node[anchor=south, yshift=0.5cm] (17) at (16.north) {$\ord(\phi(q_{ji}))+\ord(\phi(g_i))$};
      \draw[draw opacity=0] (17) -- node[sloped] {$=$} (16);

      \node[anchor=south, yshift=0.5cm, text width=5cm, text centered, font=\scriptsize] (110expl) at (110.north) {$\LM_>(g_j')\!\!\notin\!\!\LM_>(\mI\!\cdot\!\widehat M)$\\and (*)};
      \draw[->, shorten <= 3pt, shorten >= 3pt] (110expl) -- (110);
      \node[anchor=north, yshift=-0.5cm, font=\scriptsize] (140expl) at (140.south) {$\deg_w(q_{ji})=\deg_w(h_j')-\deg_w(h_i)$};
      \draw[->, shorten <= 3pt, shorten >= 3pt] (140expl) -- (140);
      \node[anchor=north, yshift=-0.5cm, font=\scriptsize] (150expl) at (150.south) {$g_{i,\ini}=h_i$};
      \draw[->, shorten <= 3pt, shorten >= 3pt] (150expl) -- (150);
    \end{tikzpicture}
  \end{center}
  in which we also use that $\LM_>(h_i')\notin \mI_\ini \cdot\widehat N$, $\LM_>(q_{ji})\notin \mI$ and $\LM_>(g_i)\notin \mI \cdot\widehat M$ for the correlation between degree resp. order and weighted degree.
\end{proof}

Now that we have this algorithm for lifting standard bases of preimages in $\widehat N$ to the corresponding preimage in $\widehat M$, we can write down our algorithm for lifting syzygies over $\Gr_\ma(A_L)$ to syzygies over $A_L$.




\begin{center}
  \begin{tikzpicture}
    \node (gr0) at (-2.5,0,5) {$\cdots$};
    \node (gr1) at (0,0,5) {$N$};
    \node (gr2) at (4,0,5) {$\bigoplus_{i=1}^k \Gr_\ma(A_L)(-\deg\theta_i)$};
    \node (al0) at (-2.5,2.5,5) {$\cdots$};
    \node (al1) at (0,2.5,5) {$M$};
    \node (al2) at (4,2.5,5) {$\bigoplus_{i=1}^k A_L(-\nu_\ma(\delta_i))$};
    \node (ky0) at (-2.5,0,0) {$\cdots$};
    \node (ky1) at (0,0,0) {$\widehat N$};
    \node (ky2) at (4,0,0) {$\bigoplus_{i=1}^k K[Y](-\deg\theta_i)$};
    \node (ay0) at (-2.5,2.5,0) {$\cdots$};
    \node (ay1) at (0,2.5,0) {$\widehat M$};
    \node (ay2) at (4,2.5,0) {$\bigoplus_{i=1}^k Q^0[Y]_>(-\nu_\ma(\delta_i))$};
    \draw[->] (gr2) -- (gr1);
    \draw[->] (gr1) -- (gr0);
    \draw[->] (al2) -- (al1);
    \draw[->] (al1) -- (al0);
    \draw[->] (ay2) -- (ay1);
    \draw[->] (ay1) -- (ay0);
    \draw[->] (ky2) -- (ky1);
    \draw[->] (ky1) -- (ky0);
    \draw[->>] (ky1) -- node[anchor=south east] {$\lambda$} (gr1);
    \draw[->>] (ky2) -- node[anchor=south east] {$\lambda$} (gr2);
    \draw[->>] (ay1) node[anchor=north east, yshift=-5pt, xshift=-15pt] {$\phi$} -- (al1);
    \draw[->>] (ay2) node[anchor=north east, yshift=-5pt, xshift=-15pt] {$\phi$} -- (al2);
    \draw[->, dotted] (al1) --  (gr1);
    \draw[->, dotted] (al2) --  (gr2);
    \draw[->, dotted] (ay1) -- (ky1);
    \draw[->, dotted] (ay2) -- (ky2);
    \fill[white] (0,3.15,5) rectangle (1.5,3.65,5);
    \node[above,yshift=1.5em,font=\footnotesize] (d1) at (al1)
      {$\phantom{\{\delta_i\}_i=1^k}\Delta=\{\delta_i\}_{i=1}^k$};
    \draw[draw opacity=0] (d1) -- node[sloped] {$\subset$} (al1);
    \node[below,yshift=-1.5em,font=\footnotesize] (e1) at (gr1)
      {$\phantom{\{\theta_i\}_{i=1}^k=}\Theta=\{\theta_i\}_{i=1}^k$};
    \draw[draw opacity=0] (e1) -- node[sloped] {$\subset$} (gr1);
    \fill[white] (4,3.15,5) rectangle (5.5,3.55,5);
    \node[above,yshift=1.5em,font=\footnotesize] (M) at (al2)
      {$\phantom{\gamma\in{}}\syz(\Delta)\ni\gamma$};
    \draw[draw opacity=0] (M) -- node[sloped] {$\subset$} (al2);
    \node[below,yshift=-1.5em,font=\footnotesize] (inM) at (gr2)
      {$\phantom{\eta\in}\syz(\Theta)\ni\eta$};
    \draw[draw opacity=0] (inM) -- node[sloped] {$\subset$} (gr2);
    \node[above,yshift=1.5em,font=\footnotesize] (g1) at (ay1)
      {$\phantom{G'\subset{}_{i=1}^k}\{g_i\}_{i=1}^k\subseteq G'$};
    \draw[draw opacity=0] (g1) -- node[sloped] {$\subset$} (ay1);
    \node[above,yshift=1.5em,font=\footnotesize] (g2) at (ay2)
      {$g$};
    \draw[draw opacity=0] (g2) -- node[sloped] {$\in$} (ay2);
    \node[below,yshift=-1.5em,font=\footnotesize] (h1) at (ky1)
      {$\phantom{H'\subset{}_{i=1}^k}\{h_i\}_{i=1}^k\subseteq H'$};
    \draw[draw opacity=0] (h1) -- node[sloped] {$\subset$} (ky1);
    \node[below,yshift=-1.5em,font=\footnotesize] (h1) at (ky2)
      {$h$};
    \draw[draw opacity=0] (h1) -- node[sloped] {$\in$} (ky2);
  \end{tikzpicture}
\end{center}

\begin{algorithm}[H]
\caption{lifting syzygies}
\label{liftingS}
\begin{algorithmic}[1]
\Require{$\eta\in\syz(\theta_1,\ldots,\theta_k)\subseteq \bigoplus_{i=1}^k \Gr_\ma(A_L)(-\deg\theta_i)$ homogeneous and
  \begin{enumerate}[leftmargin=*]
  \item $g_i \in \widehat M$ with $\LM_>(g_i)\notin\LM_>(\mI\cdot \widehat M)$ such that $\initial_\ma(\phi(g_i))=\theta_i$.
  \item a standard basis $G_0$ of $\mI\cdot \widehat M$,
  \item a set $G'=\{g_1',\ldots,g_l'\}\subseteq\widehat M$ such that $G'\cup G_0$ is a standard basis of $\phi^{-1}M'$, where $M'=\langle \theta_1,\ldots,\theta_k \rangle$ and $q_{ji}\in Q^0[Y]$ weighted homogeneous with $\LM_>(q_{ji})\notin \LM_>(\mI\cdot \widehat M)$,
    \[ g_j'=\sum_{i=1}^k q_{ji} \cdot g_i  + r_j\in\widehat M \text{ for some } r_j\in \mI\cdot \widehat M, \]
    and $\ord(\phi(g_j'))=\ord(\phi(q_{ji}'\cdot g_i))$ unless $q_{ji}'\cdot g_i=0$, as in the output of Algorithm \ref{liftingGB}.
  \end{enumerate}}
\Ensure{$\gamma\in\syz(\Delta)\subseteq \bigoplus_{i=1}^k A_L(-\ord(\delta_i))$ such that $\ini_\ma(\gamma)=\eta$.}
\algstore{liftingSyzygies}
\end{algorithmic}
\end{algorithm}

\begin{algorithm}[H]
\begin{algorithmic}[1]
\algrestore{liftingSyzygies}
\State Pick a homogeneous representative $h$ of $\eta$, say
\[h=\sum_{i=1}^k c_{i,\initial}\cdot e_i \in \bigoplus_{i=1}^k K[Y](-\deg\theta_i) \]
with $c_i\in Q^0[Y]$ and $\LM_>(c_i)\notin\LM_>(\mI)$.
\State Compute a normal form
\[ r:=\NF\Big(\sum_{i=1}^k c_i \cdot g_i,G_0\Big)\in \widehat M. \]
\State Compute a standard representation
\[ r=\sum_{j=1}^l d_j\cdot g_j' \in \widehat M. \]
\State Set
\[ g:=\sum_{i=1}^k \Big(c_i-\sum_{j=1}^l q_{j,i}\cdot d_j\Big)\cdot e_i \in \bigoplus_{i=1}^k Q^0[Y]_>(-\ord(\delta_i)). \]
\Return{$\gamma:=\phi(g)\in\bigoplus_{i=1}^k A_L(-\ord(\delta_i))$}
\end{algorithmic}
\end{algorithm}
\begin{proof} To show that $\gamma$ lies in $\syz(\Delta)$, it suffices to show that the image of $g$ in $\widehat M$ lies in $\ker(\phi)=\mI\cdot\widehat M$. That image is given by
  {\allowdisplaybreaks
  \begin{align*}
    \textstyle\sum_{i=1}^k (c_i-\sum_{j=1}^l q_{j,i}\cdot d_j)\cdot g_i  & = \textstyle\sum_{i=1}^k c_i\cdot g_i - \textstyle\sum_{i=1}^k\sum_{j=1}^l q_{j,i}\cdot d_j\cdot g_i\\
    & = \textstyle\sum_{i=1}^k c_i\cdot g_i - \sum_{j=1}^l d_j\cdot \textstyle\sum_{i=1}^k q_{j,i}\cdot g_i\\
    & = \textstyle\sum_{i=1}^k c_i\cdot g_i - \sum_{j=1}^l d_j\cdot (g_j'-r_j)\\
    & = \underbrace{\textstyle\sum_{i=1}^k c_i\cdot g_i - r}_{\in\mI\cdot\widehat M} - \textstyle\sum_{j=1}^l \underbrace{d_j\cdot r_j}_{\in \mI\cdot\widehat M},
  \end{align*}}
  as $r=\NF(\sum_{i=1}^k c_i\cdot g_i,G_0)$ and $r_j\in\mI\cdot\widehat M$ for $j=1,\ldots,l$.

  It remains to show that $\eta = \initial_\ma(\gamma)$. Note that we have for any $j=1,\ldots,l$
  \begin{center}
    \begin{tikzpicture}
      \matrix (m) [matrix of math nodes, row sep=1.5em, column sep=0em, text height=1.5ex, text depth=0.25ex]
        { \ord(\phi(r)) & = & -\deg_w(r) & \leq & -\deg_w(d_j \cdot g_j') & \leq & \ord(\phi(d_j \cdot g_j')). \\ };
      \node[anchor=north, font=\scriptsize, yshift=-0.5cm, text width=1.5cm, text centered] (eq1) at (m-1-2.south) {Step 2 and Prop. \ref{prop:AlGraal}};
      \node[anchor=north, font=\scriptsize, yshift=-0.5cm, text width=2cm, text centered] (leq1) at (m-1-4.south) {Step 3 and definition of $>$};
      \node[anchor=north, font=\scriptsize, yshift=-0.5cm, text width=2cm, text centered] (leq2) at (m-1-6.south) {always as $\phi:\! Y\!\mapsto\! f$};
      \draw[->, shorten <= 3pt, shorten >= 3pt] (eq1) -- (m-1-2);
      \draw[->, shorten <= 3pt, shorten >= 3pt] (leq1) -- (m-1-4);
      \draw[->, shorten <= 3pt, shorten >= 3pt] (leq2) -- (m-1-6);
    \end{tikzpicture}
  \end{center}

  Furthermore, observe that $\eta\in\syz(\Theta)$ implies that $\sum_{{i}=1}^k \lambda(c_{i,\ini} \cdot h_{i})=0$. Using Proposition~\ref{prop:AlGraal}, our choice of $g_i\in\widehat M$ and $c_i\in Q^0[Y]$ therefore yields
  \begin{center}
    \begin{tikzpicture}
      \matrix (m) [matrix of math nodes, row sep=2em, column sep=-0.5em, text height=1.5ex, text depth=0.25ex]
        { \sum_{i=1}^k \initial_\ma(\phi(c_i\cdot g_i)) & = & \textstyle\sum_{i=1}^k &\initial_\ma(\phi(c_i)) &\cdot&\initial_\ma(\phi(g_i))  \\
                                                        &   & \textstyle\sum_{i=1}^k &\lambda(c_{i,\ini}) &\cdot&\lambda(g_{i,\ini}) & = & \textstyle\sum_{i=1}^k \lambda(c_{i,\ini} \cdot h_{i}) = 0, \\ };
      \draw[draw opacity=0] (m-1-4) -- node[sloped, midway] (eq1) {$=$} (m-2-4);
      \draw[draw opacity=0] (m-1-6) -- node[sloped, midway] (eq2) {$=$} (m-2-6);
      \node[anchor=west, font=\scriptsize, xshift=0.5cm] (expl2) at (eq2) {$\LM_>(g_i)\notin \LM_>(\mI\cdot \widehat M)$};
      \draw[->, shorten <= 1pt, shorten >= 1pt] (expl2) -- (eq2);
      \node[anchor=east, font=\scriptsize, xshift=-0.5cm] (expl1) at (eq1) {$\LM_>(c_i)\notin \LM_>(\mI\cdot \widehat M)$};
      \draw[->, shorten <= 1pt, shorten >= 1pt] (expl1) -- (eq1);
      \node[anchor=north, font=\scriptsize, yshift=-0.35cm, xshift=0.2cm] (expl3) at (m-2-6.south) {$g_{i,\ini}=h_i$};
      \draw[->, shorten <= 1pt, shorten >= 1pt, xshift=0.2cm] (expl3) -- (m-2-6);
    \end{tikzpicture}
  \end{center}
  which implies that $\ord(\sum_{i=1}^k\phi(c_i\cdot g_i)) > \ord(\phi(c_s\cdot g_s))$ for any $s=1,\ldots,k$.

  Together, we get for all $i=1,\ldots,k$ and all $j = 1,\ldots,l$
  \[ \ord(\phi(d_j\cdot g'_j))\ge \ord(\phi(r))>\ord(\phi(c_i\cdot g_i)). \]

  Recall that by assumption (4) $\ord(\phi(g_j'))=\ord(\phi(q_{ji}\cdot g_i))=\nu_\ma(\phi(q_{ji}))+\ord(g_i)$, which implies
  \[ \ord(e_i) = \ord(\phi(g_j')) - \nu_\ma(\phi(q_{ji})). \]

  Therefore
  \begin{align*}
    \ord(\phi(q_{ji}\cdot d_j\cdot e_i)) &= \nu_\ma(\phi(d_j)) + \ord (\phi(q_{ji}\cdot e_i)) \\
                                         &= \nu_\ma(\phi(d_j)) + \ord (\phi(q_{ji}\cdot g_i)) \\
                                         &= \nu_\ma(\phi(d_j)) + \ord (\phi(g_j')) \\
                                         &= \ord(\phi(d_j\cdot g'_j)) \\
                                         &> \ord(\phi(c_i\cdot g_i)) = \ord(\phi(c_i\cdot e_i)) \\
  \end{align*}
  so that
  \[\initial_\ma(\gamma)=\initial_\ma\left(\phi\Big(\sum_{i=1}^k \Big(c_i-\sum_{j=1}^l q_{j,i}\cdot d_j\Big)\cdot e_i \Big)\right) = \initial_\ma\left(\phi\Big(\sum_{i=1}^kc_i\cdot e_i\Big)\right)=\eta. \qedhere \]
\end{proof}

Now that we are able to lift syzygies with Algorithm \ref{liftingS}, we obtain as an immediate consequence:

\begin{corollary}
Let $I\unlhd A_L$ be an ideal. For any graded, $\Gr_\ma(A_L)$-free resolution $\mathscr{C}_\bullet$
there exist a graded, $A_L$-free resolution $\mathscr{D}_\bullet$ such that the following diagram commutes:
\begin{center}
\begin{tikzpicture}[description/.style={fill=white,inner sep=2pt}]
\matrix (m) [matrix of math nodes, row sep=1.5em,column sep=0em,%
text height=1.5ex, text depth=0.25ex]
{ \mathscr{C}_\bullet: \quad 0 & \longleftarrow & A_L/I & \longleftarrow & A_L & \longleftarrow & M_1 & \longleftarrow & M_2 & \longleftarrow & \cdots\phantom{,} \\
  \mathscr{D}_\bullet: \quad 0 & \longleftarrow & \Gr_\ma(A_L)/\initial_\ma(I) & \longleftarrow & \Gr_\ma(A_L) & \longleftarrow & N_1 & \longleftarrow & N_2 & \longleftarrow & \cdots, \\ };
\path[->](m-1-3) edge node[auto] {$\initial_\ma$}(m-2-3);
\path[->](m-1-5) edge node[auto] {$\initial_\ma$}(m-2-5);
\path[->](m-1-7) edge node[auto] {$\initial_\ma$}(m-2-7);
\path[->](m-1-9) edge node[auto] {$\initial_\ma$}(m-2-9);
\end{tikzpicture}
\end{center}
where if $N_k = \bigoplus_{j\in\Z} \Gr_\ma(A_L)(-j)^{b_{j-k,k}}$ we have $M_k = \bigoplus_{j\in\Z} A_L(-j)^{b_{j-k,k}}$.

In particular, if $A_L$ is regular, there always exists a finite $A_L$-free resolution of $I$.
\end{corollary}

\begin{proof}
Suppose $\eta_1,\ldots,\eta_k \in N_i$ are the images of the canonical basis elements of $N_{i+1}$, so that they generate the syzygy module of the images of the canonical basis elements of $N_i$.

We can then apply Algorithm \ref{liftingGB} to lift them to $\theta_1,\ldots,\theta_k\in M_i$, which we will set as the images of the canonical basis elements of $M_{i+1}$. By our algorithm, they are mapped to $0$ so that our resulting sequence $\mathscr{C}_\bullet$ of $\A_L$ modules is a complex.

However, we also know that the initial forms of $\eta_1,\ldots,\eta_k$ generate the initial of the syzygy module, hence the $\eta_i$ must generate the syzygy module and our complex is exact, giving us a $A_L$-free resolution $\mathscr{D}_\bullet$ of $A_L/I$.

Clearly finite graded $\Gr_\ma(A_L)$-free resolutions of $\initial_\ma(I)$ lifts to finite $A_L$-free resolutions.
If $\Gr_\ma(A_L)$ is isomorphic to a polynomial ring, i.e.~ $\Gr_\ma(A_L)$ is regular by Proposition \ref{prop:regularity},  it follows from Hilbert's Syzygy Theorem that a finite free-resolution of $\initial_\ma(I)$ exists.
\end{proof}

\begin{example}
  Consider the union of a circle on a hyperplane and the intersection of a twisted cubic with that hyperplane
  \[ I:=\langle (x-1)^2+y^2-3,z \rangle \cap \langle xz-y^2,yw-z^2,xw-yz,z\rangle \unlhd A:=\Q[x,y,z,w]. \]
  The twisted cubic makes the resolution of $I$ more complicated, a minimal resolution would be of the form
  \[ 0 \longleftarrow A/I \longleftarrow A \longleftarrow A^4 \longleftarrow A^5 \longleftarrow A^2 \longleftarrow 0. \]
  However if we localize at a subvariety outside the twisted cubic, say on two conjugate points on the circle,
  \[ L:=I + \langle x-1 \rangle = \langle z, x-y, y^2w-3w,y4-3y2 \rangle \unlhd A, \]
  the local ring will be of the form
  \[ A_L = \bigslant{\Q(w)[Y_1,\ldots,Y_s,x,y,z]}{\langle f_1-Y_1, \ldots, f_4-Y_4, 3wY_3-w^2Y_4+Y_3^2\rangle}, \]
  where $f_1,\ldots,f_4$ are the four generators of $L$ stated above.

  The associated graded ring of $A_L$ is then isomorphic to
  \[ \Gr_\ma (A_L) = K[Y_1,Y_2,Y_3,Y_4]/\langle 3\cdot Y_3-w\cdot Y_4 \rangle, \]
  where
  \begin{align*}
    K=\Quot(A/L)&\cong \Q(w)[t]/(t^2-4t-71), \\
       x,\,y,\,z=0&\mapsto 1,\,\textstyle\frac{3-t}{5},\,0,
  \end{align*}
  and the initial ideal of $I$ has the simple form of
  \[ \initial_\ma(I)=\langle w^2Y_4, Y_1 \rangle \unlhd \Gr_\ma(A_L). \]

  It is easy to see that $\initial_\ma(I)$ allows for a Koszul resolution
  \[ 0 \longleftarrow \Gr_\ma (A_L)/\initial_\ma(I) \longleftarrow \Gr_\ma (A_L) \overset{\left(\begin{smallmatrix} w^2Y_4\\Y_1\end{smallmatrix}\right)}{\longleftarrow} \Gr_\ma (A_L)^2 \overset{\left(\begin{smallmatrix} -Y_1 & w^2Y_4 \end{smallmatrix}\right)}{\longleftarrow} \Gr_\ma (A_L) \longleftarrow 0, \]
  which then lifts to an equally simple resolution
  \[ 0 \longleftarrow A_L/ I \longleftarrow A_L \overset{M_1}{\longleftarrow} A_L^2 \overset{M_2}\longleftarrow A_L \longleftarrow 0,  \]
  but with more complicated matrices
  \[ M_1=\begin{pmatrix} w^2Y_4+3w^2xY_2^2+3wY_2Y_3-Y_3^2 & Y_1 \end{pmatrix} \]
  and
  \[ M_2=\begin{pmatrix} -Y_1 \\ w^2 Y_4+3w^2Y_2^2+w^2Y_2Y_4-\frac{w^2}{9}Y_4^2+3w^2Y_2^3-Y_2Y_3^2+\frac{1}{3w}Y_3^3+\frac{1}{9}Y_3^2Y_4 \end{pmatrix}. \]
\end{example}

\section{acknowledgements}
We would like to thank Janko Boehm for insightful converstations and Theo Mora for clarifications of the current state of research.

\end{document}